\documentclass[12pt, reqno]{amsart}
\usepackage{graphicx} % Required for inserting images
\usepackage[margin=1in]{geometry}            
\usepackage{graphicx}
\usepackage{amsmath, amssymb,amscd,amsthm}
\usepackage{epstopdf}
\usepackage{tikz}
\usepackage[utf8]{inputenc}
\usepackage{xcolor}
\usepackage{soul,dirtytalk}
\usepackage[colorlinks=true, pdfstartview=FitV, linkcolor=blue, citecolor=blue, urlcolor=blue]{hyperref}
\usepackage{tikz-cd}
\usepackage[nameinlink]{cleveref}
\usepackage{comment}
\usepackage[colorinlistoftodos]{todonotes}

\title{Lifting Frobenius splittings through geometric vertex decomposition}

\author[De Negri]{Emanuela De Negri}
\address{Dipartimento di Matematica, 
Universit\`a di Genova,
Genova, Italy}
\email{denegri@dima.unige.it}
\author[Gorla]{Elisa Gorla}
\address{Department of Mathematics, University of Neuchatel, Neuchatel, Switzerland}
\email{elisa.gorla@unine.ch}
\author[Klein]{Patricia Klein}
\address{Department of Mathematics, Texas A\&M University, College Station, TX, USA}
\email{pjklein@tamu.edu}
\author[Rajchgot]{Jenna Rajchgot}
\address{Department of Mathematics and Statistics, McMaster University, Hamilton, ON, Canada}
\email{rajchgoj@mcmaster.ca}
\author[Seccia]{Lisa Seccia}
\address{Department of Mathematics, University of Neuchatel, Neuchatel, Switzerland}
\email{lisa.seccia@unine.ch}

\thanks{Klein was partially supported by NSF DMS-2246962. Rajchgot was partially supported by NSERC Discovery Grants 2017-05732 and 2023-04800. Seccia was partially supported by SNSF grant TMPFP2\_217223. This project began at the workshop Women in Commutative Algebra II, for which funding was provided by the Association for Women in Mathematics, the Atlantic Association for Research in the Mathematical Sciences, the Centre International de Mathématiques Pures et Appliquées, the Clay Mathematics Institute, the Committee for Developing Countries of the European Mathematical Society, and NSF grant DMS–2324929. Both hospitality and financial support were provided by the Centro Internazionale per la Ricerca Matematica in Trento, Italy.}
\date{\today}

\newtheorem{Theorem}{Theorem}[section]

\newtheorem{rmk}[Theorem]{Remark}

\newtheorem{defn}[Theorem]{Definition}
\newtheorem{prop}[Theorem]{Proposition}
\newtheorem{thm}[Theorem]{Theorem}
\newtheorem*{thm*}{Theorem}

\newtheorem*{MainThm}{Main Theorem}
\theoremstyle{remark}
\newtheorem{examplex}[Theorem]{Example}
\newenvironment{ex}
  {\pushQED{\qed}\examplex}
  {\popQED\endexamplex}

\newcommand{\init}{\mathrm{in}}
\newcommand{\hgt}{\mathrm{ht}}
\newcommand{\Tr}{\mathrm{Tr}}
\newcommand{\ass}{\mathrm{Ass}}

\DeclareMathOperator{\del}{del}
\DeclareMathOperator{\lk}{lk}

\definecolor{darkblue}{rgb}{0.0,0,0.7}

\newcommand{\newword}[1]{\textcolor{darkblue}{\textbf{\emph{#1}}}}

\begin{document}

\begin{abstract}
Frobenius splitting, pioneered by Hochster and Roberts in the 1970s and Mehta and Ramanathan in the 1980s, is a technique in characteristic $p$ commutative algebra and algebraic geometry used to control singularities.  In the aughts, Knutson showed that Frobenius splittings of a certain type descend through Gr\"obner degeneration of a certain type, called geometric vertex decomposition. In the present paper, we give a partial converse to Knutson's result.  We show that a Frobenius splitting that compatibly splits both link and deletion of a geometric vertex decomposition can, under an additional hypothesis on the form of the splitting, be lifted to a splitting that compatibly splits the original ideal. We discuss an example showing that the additional hypothesis cannot be removed. Our argument uses the relationship between geometric vertex decomposition and Gorenstein liaison developed by Klein and Rajchgot.  Additionally, we show that Li's double determinantal varieties defined by maximal minors are Frobenius split.
\end{abstract}

\maketitle

\section{Introduction}

Frobenius splitting is a tool in characteristic $p>0$ commutative algebra and algebraic geometry that has played a central role in controlling singularities over the past five decades.  Hochster and Roberts \cite{HR74} used splittings of the Frobenius morphism to prove the famous Hochster-Roberts Theorem, which states that the ring of invariants of the action of a linearly reductive group on a regular ring is Cohen--Macaulay.  Later, Mehta and Ramanathan \cite{MR85} named and further studied Frobenius splitting, establishing foundational results on vanishing of cohomology for Frobenius split varieties, giving criteria for splitting, and showing that Schubert varieties are Frobenius split.

Fedder \cite{Fed83} and Singh \cite{Sin99} showed that the property of being F-split, in contrast to other desirable algebraic properties such as being Cohen--Macaulay or Gorenstein, does not deform.  However, it does have some favorable interaction with degeneration.  Knutson \cite{Knu09}, building on work of Lakshmibai, Mehta, and Parameswaran \cite{LMP98}, studied Frobenius splitting in relation to Gr\"obner degeneration.  

If $\varphi$ is a splitting of the ring $R$ with ideal $I$ so that $\varphi(I) \subseteq I$, we say that $\varphi$ compatibly splits $I$.  For a polynomial $f$, let $\init_{x_n}(f)$ denote the sum of the terms of highest degree in $x_n$, and let $\init_{x_n}(I) = (\init_{x_n}(f) \mid f \in I)$.  The following is part of \cite[Theorem 2]{Knu09}:

\begin{thm*}[\cite{Knu09}]
    Let $f \in S=\kappa[x_1, \ldots, x_n]$ be of degree $n$, and let $<$ be a lexicographic term order with $x_n$ largest.  If $\init_<(f) = x_1 \cdots x_n$, then $f$ induces a splitting of $S$.  If this splitting compatibly splits $I$, then $\init_{x_n}(f)$ induces a splitting that compatibly splits $\init_{x_n}(I)$, $\init_{x_n}(I):x_n$, and $\init_{x_n}(I)+(x_n)$. 
\end{thm*}

We will make \say{induces a splitting} precise after defining the trace map (see \Cref{def:traceMap}).  That is, Frobenius splittings of a certain form respect lexicographic Gr\"obner degeneration. Specifically, this theorem shows that Frobenius splittings of an appropriate form descend through a geometric vertex decomposition, as defined in \cite{KMY09} (see \Cref{subs:gvd}). 

The goal of the present paper is to give a partial converse to \cite[Theorem 2]{Knu09} and to describe an obstruction to a full converse. Fix an ideal $I$ of $S$, and set $C = \init_{x_n}(I):x_n^\infty$.  Let $N'$ be the contraction of $\init_{x_n}(I)+(x_n)$ to $\kappa[x_1, \ldots, x_{n-1}]$ and $N = N'S$. Our main theorem is the following:

\begin{MainThm}[\Cref{thm:main}]
    Suppose that $\init_{x_n}(I) = C \cap (N+(x_n))$ where $C \neq S$ and $C$ and $N$ share no minimal prime.  If $x_ng$ induces a splitting that compatibly splits both $C$ and $N$ and there exists $u \in C$ that does not divide zero modulo $N$ such that $u \mid g$, then there exists $r$ such that $x_ng+r$ induces a splitting that compatibly splits $I$.
\end{MainThm}

The hypothesis that $C$ and $N$ share no minimal prime is automatic if $I$ is unmixed (i.e., all associated primes of $I$ have the same height).

This theorem is saying that, if one has a Frobenius splitting which simultaneously compatibly splits the components obtained from a geometric vertex decomposition, one can lift the splitting to a different splitting that compatibly splits the ideal that was decomposed, provided the original splitting has an appropriate form.  In \Cref{ex:u-divides-g-necessary}, we show that the assumption on the existence of $u$ as in the theorem statement cannot be dropped.  Our construction relies on the relationship between geometric vertex decomposition and Gorenstein liaison developed by Klein and Rajchgot \cite{KR21}.  The element $u$ is essential to the construction of the isomorphism that, in appropriate circumstances, makes up an elementary G-biliaison, which corresponds to a pair of Gorenstein links (see \cite{Har07}).  From this perspective, we may view the relationship between geometric vertex decomposition and Frobenius splitting as mediated by the isomorphisms that constitute biliaisons.

Structure of the paper: In \Cref{sect:Background}, we review standard material on the Frobenius morphism and Frobenius splittings as well as material on geometric vertex decomposition and its connection to Gorenstein liaison. \Cref{sect:proof-of-main-thm} is devoted to the main theorem and \Cref{sect:examples} to examples of its implementation. 

In \Cref{sect:DDV}, we study double determinantal varieties defined by maximal minors and show that they are Frobenius split. The construction of the splitting that we give follows the construction in this paper's main theorem, though the main theorem itself cannot directly be applied. One purpose of this section is to showcase a phenomenon that is closely related to the main theorem, but not explained by it. We leave the project of giving a refinement of the theory developed here which gives a satisfying explanation of this example as an open problem.

\section{Background and preliminaries}\label{sect:Background}

Throughout this paper, we fix a perfect field $\kappa$ of prime characteristic $p$ and polynomial ring $S = \kappa[x_1, \ldots, x_n]$.

\subsection{Frobenius splitting and Knutson ideals}

Throughout this section, we will assume that $R$ is a ring of prime characteristic $p$. Every ring of characteristic $p$ is equipped with a ring endomorphism called the \newword{Frobenius endomorphism} $F: R\longrightarrow R$, defined by $F(r) = r^p$. Note that the Frobenius map gives $R$ the structure of a module over itself with the non-standard action defined by $r\cdot x= r^p x$. To avoid confusion, it is customary to denote this $R$-module by $F_* R$.
 
The algebraic properties of the Frobenius map reflect the geometry of $R$. For example, it is easy to see that $F$ is injective if and only if $R$ is reduced. A fundamental theorem by Kunz \cite{Kun69} states that $F$ is flat if and only if $R$ is regular. If, for example, $R$ is a quotient of a polynomial ring over a perfect field localized at a prime ideal, then $F_*R$ is a finite $R$-module and Kunz's theorem tells us that $F_*R$ is free if and only if $R$ is regular. A weaker property than being a free $R$-module is containing at least one copy of $R$ as a direct summand.

\begin{defn}
$R$ is said to be \newword{F-split} if the Frobenius map splits in the category of $R$-modules, i.e., if there exists a homomorphism $\varphi: F_*R \rightarrow R$ of $R$-modules such that $\varphi \circ F= 1_R$. Such a $\varphi$ is called a \newword{Frobenius splitting} (or \newword{F-splitting}) of $R$.
\end{defn}

The presence of even one copy of $R$ as a direct summand of $F_*R$ turns out to do a great deal of work to control singularities.  For example, if $R$ is F-split, then $R$ is weakly normal. If $R$ is the homogeneous coordinate ring of the projective variety $X$ and $R$ is Frobenius split, then Kodaira vanishing holds for $X$. See \cite{Fed83, BK05, Sch10} for further information.

\begin{defn}\label{def:traceMap}
Considering the polynomial ring $S$, let $\Tr : S \rightarrow S$ be the linear map defined on a monomial $m$ as 
  $$  \Tr(m)= 
\begin{cases}
\dfrac{\sqrt[p]{m x_1 \cdots x_n}}{x_1 \cdots x_n} \qquad \textit{if }m \prod x_i \textit{ is a $p$-th power}\\
0 \qquad \qquad \qquad \qquad\textit{otherwise}.
\end{cases}$$
 We call $\Tr:S \rightarrow S$ the \newword{trace map}. $\Tr ((x_1 \cdots x_n)^{p-1}\bullet)$ is a Frobenius splitting of $S$, known as the \newword{standard splitting}.
\end{defn}

Given a Frobenius splitting $\varphi$ of a ring $R$, it is natural to consider those ideals $I$ such that $\varphi$ descends to a Frobenius splitting of $R/I$.  An equivalent condition is $\varphi(I) \subseteq I$. Such an ideal is said to be \newword{compatibly split} by $\varphi$.

\begin{prop}[\cite{BK05}]\label{prop-intersect-decompose} 
Let $R$ be an F-split ring with Frobenius splitting $\varphi$. Then:
 \begin{itemize}
 \item $R$ is reduced.
 \item If $\varphi$ compatibly splits the proper ideal $I$, then $\varphi (I)=I$ and $I$ is radical.
 \item If the ideals $I$ and $J$ are compatibly split by $\varphi$, then so are $I+J$ and $I \cap J$.
 \item If the ideal $I$ is compatibly split by $\varphi$, then, for every ideal $J$, $I:J$ is compatibly split by $\varphi$. In particular, the minimal primes of $I$ are compatibly split by $\varphi$.
 \end{itemize}
\end{prop}

 When $R$ is a Noetherian ring, only finitely many ideals of $R$ are compatibly split with respect to a given Frobenius splitting $\varphi$ \cite{MR85, Schw09}. Knutson \cite{Knu09} later addressed the problem of explicitly listing these ideals for classical splittings. \Cref{prop-intersect-decompose} allows for an algorithmic approach: Starting from known compatibly split ideals, one can find others by taking minimal primes, sums, and intersections. In certain cases, this procedure yields all $\varphi$-compatibly split ideals. This motivated the following definition, formally introduced by Conca and Varbaro \cite{CV20}.

\begin{defn}
Let $f\in S$ be a polynomial such that its leading term $\init_<(f)$ with respect to some term order $<$ is $x_1 \cdots x_n$. Define $\mathcal{C}_f$ to be the smallest set of ideals satisfying the following conditions:
    \begin{enumerate}
    \item $(f)\in \mathcal{C}_f$,
    \item If $I \in \mathcal{C}_f$, then $I:J\in \mathcal{C}_f$ for every ideal $J$ of $S$,
    \item If $I,J \in \mathcal{C}_f$, then $I+J \in \mathcal{C}_f$ and $I\cap J \in \mathcal{C}_f$. 
    \end{enumerate}
    When $I \in \mathcal{C}_f$, we say that $I$ is a \newword{Knutson ideal} of $f$.
\end{defn}
Using \Cref{prop-intersect-decompose}, we see that all Knutson ideals of $f$ are compatibly split by $\Tr(f^{p-1} \bullet)$.  The condition on $\init_<(f)$ is natural in the sense that one can reduce to this case (see \cite[Section 1.2]{Knu09}).

Rather than focusing solely on the existence of a Frobenius splitting, we are interested in explicitly constructing them in order to study their compatibly split ideals, which enjoy desirable relations with each other. For example, if $I, J \in \mathcal{C}_f$, then $\init_<(I)+\init_<(J) = \init_<(I+J)$ \cite{Knu09}.  See also \cite{Se21} for the result over not-necessarily-perfect fields.

\subsection{Geometric vertex decomposition}\label{subs:gvd}

In this section, we will review geometric vertex decomposition, introduced by Knutson, Miller, and Yong \cite{KMY09}.  We refer the reader to \cite{KR21} for the  relationship between geometric vertex decomposition and Gorenstein liaison.  We direct the reader to these two sources for further context and information and, specifically, to \cite[Section 2]{KMY09} and \cite[Section 2]{KR21}.

Fix a variable $y = x_i$ of $S$ and an ideal $I\subseteq S$. Consider the weight vector $w$ that is the $i$-th standard basis vector, and define $\init_{y}(I)$ to be the (not necessarily monomial) ideal obtained by degeneration with respect to that weight vector. Set \[
C(y,I) = \init_y(I):y^\infty \qquad \mbox{and} \qquad N(y,I) = (f \in I \mid \mbox{ no term of $f$ is divisible by $y$}).
\]  If \begin{enumerate}
\item $\init_y(I) = C(y,I) \cap (N(y,I)+(y))$, and 
\item either $\sqrt{C(y,I)} = \sqrt{N(y,I)}$ or no minimal prime of $C(y,I)$ is a minimal prime of $N(y,I)$,
\end{enumerate}
then we say that $I$ admits a \newword{geometric vertex decomposition} at $y$, or that $\init_y(I) = C(y,I) \cap (N(y,I)+(y))$ is a geometric vertex decomposition of $I$ with respect to $y$. In this case, we will refer to $C(y,I)$ as the \newword{link} of $I$ at $y$ and to $N(y,I)$ as the \newword{deletion}.

If $I$ is unmixed, then condition (2) is automatically satisfied.  Condition (1) holds if and only if $I$ possesses a generating set in which each term is at most linear in $y$.  Then also $\init_y(I) = yC(y,I) + N(y,I)$.
 If $I$ is a squarefree monomial ideal, then $I$ admits a geometric vertex decomposition at $y$ if and only if the Stanley--Reisner complex of $I$ admits a vertex decomposition (in the non-pure sense) at the variable corresponding to $y$. In this case, $C(y,I)$ and $N(y,I)$ will be the ideals of the link and deletion, respectively, taken in the contracted polynomial ring which omits the generator $y$.

Considering again the weight vector $w$ that is the $i$-th standard basis vector, $\init_y(I) = C(y,I) \cap (N(y,I)+(y))$ if and only if $I$ has a $w$-Gr\"obner basis $\mathcal{G}$ of the form \[
\mathcal{G} = \{yq_1+r_1, \ldots, yq_k+r_k, h_1, \ldots, h_\ell\}
\] where $y$ does not divide any term of $q_j$, $r_j$, or $h_j$. Then $C(y,I) = (q_1, \ldots, q_k, h_1, \ldots, h_\ell)$ and $N(y,I) = (h_1, \ldots, h_\ell)$. If $<$ is a term order which refines $w$, i.e., such that $\init_<(I) = \init_<(\init_y(I))$, for example, any purely lexicographic order in which $y$ is the largest variable, and $\mathcal{G}$ is a Gr\"obner basis with respect to $<$, then the given generating sets for $C(y,I)$ and $N(y,I)$ are $<$-Gr\"obner bases as well.

If $I$ admits a geometric vertex decomposition at $y$, we call the geometric vertex decomposition \newword{degenerate} if $C(y,I) = (1)$ or if $\sqrt{C(y,I)} = \sqrt{N(y,I)}$ and \newword{nondegenerate} otherwise.  If $I$ is radical, then $N(y,I)$ is radical.  In this case, $C(y,I) = (1)$ if and only if $y$ belongs to the support of some linear form in $I$, and $\sqrt{C(y,I)} = \sqrt{N(y,I)} = N(y,I)$ if and only if $I$ has a generating set that does not involve $y$.  If $I$ is a squarefree monomial ideal, then degenerate geometric vertex decompositions correspond to deconing the associated Stanley--Reisner complex or to taking a vertex decomposition at an element of the ambient set that is not a vertex of the complex.  If $I$ is unmixed and the decomposition is nondegenerate, then $\sqrt{C(y,I)}$ is unmixed and $\hgt(I) = \hgt(C(y,I)) = \hgt(N(y,I))+1$.  

Our motivating examples in this paper will be homogeneous, unmixed ideals that can be understood via a sequence of geometric vertex decompositions that are compatible with a lexicographic term order.  The precise condition follows:

\begin{defn}
With $m \geq 0$, let $I$ be an ideal of $\kappa[x_1, \ldots, x_m]$. We say that $I$ is \newword{lex-compatibly geometrically vertex decomposable} if $I$ is unmixed and if 
\begin{enumerate}
\item $I=(1)$ or $I = 0$ or $I$ is generated by a subset of $\{x_1, \ldots, x_m\}$, or
\item $I$ admits a geometric vertex decomposition at $x_m$ and the contractions of $N(x_m,I)$ and $C(x_m,I)$ to the ring $\kappa[x_1, \ldots, x_{m-1}]$ are lex-compatibly geometrically vertex decomposable.
\end{enumerate}
\end{defn}

If $<$ is the lexicographic order on $x_1<x_2<\cdots<x_m$, then $I$ is lex-compatibly geometrically vertex decomposable if and only if $\init_<(I)$ is vertex decomposable (in the pure sense) with decompositions taken in the order specified by $<$, working greatest to least.   See \cite{KR21} for the definition and properties of geometrically vertex decomposable ideals, without the requirement of compatibility with one fixed ordering of the variables.  If $I$ is geometrically vertex decomposable (whether or not lex-compatibly), then $I$ is radical.  If moreover $I$ is homogeneous, then $I$ is a Cohen--Macaulay ideal.

\section{Proof of main theorem}\label{sect:proof-of-main-thm}

Our goal in this section is to  prove that a nondegenerate geometric vertex decomposition can, in appropriate circumstances, be used to lift a Frobenius splitting of the link and deletion to a Frobenius splitting of the ideal being decomposed (\Cref{thm:main}).  In order to give this construction, we will make use of \cite[Lemma 7.4]{KR21}, which produces a module isomorphism from a geometric vertex decomposition.  The motivation for considering this isomorphism comes from the theory of Gorenstein liaison.  It is the isomorphism itself rather than liaison theory that we will use to prove the main theorem.

The statement within \cite{KR21} has the hypotheses of a homogeneous ideal within a polynomial ring over an infinite field, and the conclusion is stated with a graded isomorphism. Here we require neither the homogeneity assumption nor the assumption of an infinite field (which, in retrospect, could have been avoided in \cite{KR21}). Below will state the version of this result that we will use here with a self-contained proof.

\begin{prop}[{\cite[Lemma 7.4, essentially]{KR21}}]\label{prop:non-homogeneous--mixed-KR}
Let $I$ be an ideal of $S$, and suppose that $\init_{x_n}(I) = C(x_n,I) \cap (N(x_n,I)+(x_n))$ is a nondegenerate geometric vertex decomposition. If $N(x_n,I)$ has no embedded primes, then 
\begin{enumerate} 
\item There exist $q \in C(x_n,I)$ and $x_nq+r \in I$ both nonzerodivisors modulo $N(x_n,I)$ with no term of $q$ or $r$ divisible by $x_n$.
\item Fix $q$ and $r$ as in Part (1). For each $f \in C(x_n,I)$, there exists $g_f\in I$ such that $f(x_nq+r)-g_fq \in N(x_n,I)$. In addition, if no term of $f$ is divisible by $x_n$, then one can choose $g_f=x_nf+s$, where no term of $s$ is divisible by $x_n$.
\item For $g_f$ as in Part (2), $\overline{g_f} \in S/N(x_n,I)$ is uniquely determined and the assignment $f \mapsto \overline{g_f}$ induces an isomorphism \[
\psi:C(x_n,I)/N(x_n,I) \rightarrow I/N(x_n,I).
\]
\end{enumerate}
\end{prop}

\begin{proof}
Write $C = C(x_n,I)$ and $N = N(x_n,I)$.

\noindent \textbf{(1)} Let $<$ be a term order such that $\init_<(I) = \init_<(\init_{x_n}(I))$. Write the reduced $<$-Gr\"obner basis of $I$ as $\mathcal{G} = \{x_nq_1+r_1, \ldots, x_nq_k+r_k, h_1, \ldots, h_\ell\}$ such that no term of any $q_i$, $r_i$, or $h_i$ is divisible by $x_n$. Set $S' = \kappa[x_1, \ldots, x_{n-1}]$, and define the $S'$-module $X = \mathrm{span}_{S'}\{q_i \mid i \in [k]\}$.

Suppose for the sake of contradiction that $X \subseteq \bigcup_{P \in \ass(N)} P$. Because each $h_j \in N \subseteq \bigcap_{P \in \ass(N)} P$, we have also $\mathrm{span}_{S'}\{q_i , h_j\mid i \in [k], j \in [\ell]\} \subseteq \bigcup_{P \in \ass(N)} P$. Because $C$ is generated by the $q_i$ and the $h_j$, none of which involve $x_n$, we may write \[
C = \bigoplus_{t \geq 0} x_n^t \cdot \mathrm{span}_{S'}\{q_i , h_j\mid i \in [k], j \in [\ell]\}.
\] Thus $C \subseteq \bigcup_{P \in \ass(N)} P$. By the Prime Avoidance Lemma, $C$ must be contained in a single prime $P \in \ass(N)$. By assumption, $N$ has no embedded primes. Hence $P$ is a minimal prime of $N$. Because $N \subset C \subseteq P$ and $P$ is minimal over $N$, $P$ must be minimal over $C$ as well.  But then $C$ and $N$ share a minimal prime, which contradicts the assumption that the decomposition is nondegenerate. 

Thus, we can fix an element $q \in X \setminus \bigcup_{P \in \ass(N)} P$. Then $q \in C$, and $q$ is not a zerodivisor modulo $N$. Write $q = s_1q_1+\cdots+s_kq_k$, $s_i \in S'$. Set $r = s_1r_1+\cdots+s_kr_k$. Then no term of $q$ or of $r$ is divisible by $x_n$, and $x_nq+r = \sum_{i=1}^k s_i(x_nq_i+r_i) \in I$.  Because $N$ has a generating set that does not involve $x_n$, so too does each of its associated primes. Hence $x_nq+r$ cannot be an element of any associated prime of $N$ unless $q$ is. Since $q$ is not, we conclude that $x_nq+r$ is not a zerodivisor modulo $N$.

\vspace{2mm}
\noindent \textbf{(2)}
Choose any expression of $f$ as an $S$-linear combination of the generators of $C$: \begin{equation}\label{eqn:f}
f = a_1q_1+\cdots+a_kq_k+b_1h_1+\cdots +b_\ell h_\ell \in C.
\end{equation}
Set \[
g_f = a_1(x_nq_1+r_1)+\cdots+a_k(x_nq_k+r_k) \in I.
\]  Now \[
f(x_nq+r)-g_fq = a_1(rq_1-qr_1)+\cdots+a_k(rq_k-qr_k)+(b_1h_1+\cdots +b_\ell h_\ell)(x_nq+r).
\] 
   But each $rq_i-qr_i = (x_nq+r)q_i-(x_nq_i+r_i)q \in (x_nq+r, x_nq_i+r_i) \subseteq I$ and has no term divisible by $x_n$, hence is an element of $N$. Also each $h_j \in N$, and so $f(x_nq+r)-g_fq \in N$, as desired.

If in addition no term of $f$ is divisible by $x_n$, in \Cref{eqn:f} one can choose $a_i$ and $b_j$ so that none of their terms involves $x_n$.
Then \[
g_f = x_n(a_1q_1+\cdots+a_kq_k)+(a_1r_1+\cdots+a_kr_k)=x_n(f-(b_1h_1+\cdots +b_\ell h_\ell))+s,
\] 
where $s = a_1r_1+\cdots+a_kr_k$ does not involve $x_n$. Since $f(x_nq+r)-g_fq\in N$ and $(x_nf+s)-g_f=x_n(b_1h_1+\cdots +b_\ell h_\ell)\in N$, then also $f(x_nq+r)-(x_nf+s)q\in N$. Moreover, since $g_f \in I$ and $x_nf+s -g_f \in N \subset I$, it follows that $x_nf+s \in I$. Therefore, $g_f$ may be replaced by $x_nf+s$, completing the proof.

\vspace{2mm}
\noindent \textbf{(3)} Suppose that there exist $g_f$ and $g'_f \in S$ so that
$f(x_nq+r)-g'_fq,  f(x_nq+r)-g_fq \in N$. Then \[
(g_f - g'_f)q  = f(x_nq+r)-g'_fq-(f(x_nq+r)-g_fq)\in N.
\] Since $q$ is a nonzerodivisor modulo $N$ from Part (1), we must have $g_f - g'_f \in N$, that is $\overline{g_f} = \overline{g'_f} \in S/N$, proving that $\overline{g_f}\in S/N$ in uniquely determined by $f$.

Uniqueness of $\overline{g_f}\in S/N$ together with the hypothesis $g_f \in I$ from Part (2) implies that the assignment $g \mapsto \overline{g_f}$ gives a well-defined map $\psi: C/N \rightarrow I/N$, which one sees easily is a map of $S$-modules. 
 
To see injectivity, suppose $\psi(\overline{f}) = \overline{0}$.  Then $g_f \in N$. Because $f(x_nq+r) - g_fq \in N$, we have $f(x_nq+r) \in N$. Since $x_nq+r$ is a nonzerodivisor modulo $N$ from Part (1), it follows that $f \in N$, which is to say that $\overline{f} = \overline{0}$.
    
To see surjectivity, because $I$ is generated over $N$ by the $x_nq_i+r_i$, $i \in [k]$, it suffices to show that the class of each $x_nq_i+r_i$ in $I/N$ is in the image of $\psi$. By construction, each such class is the image of the class of $q_i$ in $C/N$.
\end{proof}

\begin{rmk}\label{rmk:iso}
We often describe the map $\psi$ as multiplication by $(x_n q+r)/q$. If $f\nmid 0$ modulo $N$ and $f$ does not involve $x_n$, then $f(x_n q+r)-(x_n f+s)q\in N$ tells us that ($x_n f+s)/f$ and $(x_n q+r)/q$ denote the same fraction modulo $N$. From this perspective, we interpret Part (2) of \Cref{prop:non-homogeneous--mixed-KR} as saying that we may also describe $\psi$ as multiplication by $(x_nf+s)/f$. 

Notice moreover that, given $c\in C$ and $i\in S$ such that $c(x_n q+r)-iq\in N$, one has $i\in I$ and $\psi(\overline{c})=\overline{i}$. This is a consequence of the uniqueness in Part (3).
\end{rmk}

The following is a partial converse to \cite[Theorem 2]{Knu09} and the main result of this paper.

\begin{thm}\label{thm:main}
    Let $I$ be an ideal of $S$, and suppose that $\init_{x_n}(I) = C(x_n,I) \cap (N(x_n,I)+(x_n))$
    is a nondegenerate geometric vertex decomposition. Assume that $g \in S$ has no term divisible by $x_n$ and that $\Tr((x_ng)^{p-1}\bullet)$ is a Frobenius splitting of $S$ that simultaneously compatibly splits $C(x_n,I)$ and $N(x_n,I)$.  Suppose further that there exists $u \in C(x_n,I) \setminus \bigcup_{P \in \ass(N(x_n,I))} P$ such that $u \mid g$.  Then there exists $r \in S$ with no term divisible by $x_n$ such that $\Tr((x_ng+r)^{p-1} \bullet )$ compatibly splits $I$.  
    
    Specifically, $x_ng+r$ is a representative of the image of the class of $g$ under the isomorphism $\psi$ of Part (3) of \Cref{prop:non-homogeneous--mixed-KR}.
\end{thm}

\begin{proof}
Write $C=C(x_n,I)$ and $N = N(x_n,I)$.  Because $N$ is compatibly split, $N$ must be radical and hence has no embedded primes.  Let $\psi:C/N \rightarrow I/N$ be the isomorphism guaranteed by Part (3) of \Cref{prop:non-homogeneous--mixed-KR}.  Take $u \mid g$ as in the hypotheses, and set $y = x_n$. 

By Part (2) of \Cref{prop:non-homogeneous--mixed-KR}, there exists $s\in S$ such that $yu+s \in I$ and no term of $s$ is divisible by $y$. By Parts (1) and (3) of \Cref{prop:non-homogeneous--mixed-KR}, $\psi$ is multiplication by $(yu+s)/u$. Set $f = (yu+s)g/u$ (which is an equality in $S$ because $u \mid g$). Notice that $f$ is of the form $yg+r$ where no term of $r$ is divisible by $y$.  Thus, $\init_{y}(f) = yg$, and so $\Tr(f^{p-1} \bullet)$ splits $S$ because $\Tr((yg)^{p-1} \bullet)$ does by \cite[Theorem 2(1)]{Knu09}. 

Fix $i \in I$ and let $v=yu+s$. By Part (3) of \Cref{prop:non-homogeneous--mixed-KR}, $\overline{i}=\psi(\overline{c})$ for some $c \in C$. Since $\psi$ is multiplication by $v/u$, $iu = vc$ modulo $N$, i.e.,
we may write $iu = vc+m$ for some $m \in N$.  Then
    \begin{align*}
    u \Tr(f^{p-1}i) & = \Tr(u^pf^{p-1}i)\\
    &=\Tr((uf)^{p-1}(iu))\\
    &=\Tr((vg)^{p-1}(vc+m))\\
     & = \Tr(v^pg^{p-1}c+v^{p-1}g^{p-1}m)\\
     & = v\Tr(g^{p-1}c)+\Tr(v^{p-1}g^{p-1}m).
    \end{align*}
We claim that $\Tr(g^{p-1}c) \in C$ and that $\Tr(v^{p-1}g^{p-1}m) \in N$.
    
    By additivity of the trace map and the fact that $C$ has a generating set that does not involve $y$, we may assume that $c = \mu c'$ for some monomial $\mu$ and some $c' \in C$ that does not involve $y$.  If there exist $k_i \in \mathbb{Z}_{\geq 0}$ such that $\mu = y^{p-1}y^{k_0p}x_1^{k_1p}\cdots x_n^{k_np}$, then $\Tr(g^{p-1}c) = y^{k_0}x_1^{k_1}\cdots x_n^{k_n}\Tr((yg)^{p-1}c') \in C$ by the assumption that $\Tr((yg)^{p-1} \bullet)$ compatibly splits $C$. Otherwise, clearly $\Tr(g^{p-1}c) = 0 \in C$.

    Similarly, because $N$ has a generating set that does not involve $y$, using additivity of the trace map, it suffices to consider $\Tr(g^{p-1}(yu+s)^{p-1}\nu m')$ where $\nu$ is a monomial and $m' \in N$ is a polynomial with no term divisible by $y$.  Let $d$ be the greatest power of $y$ that divides $\nu$, which we may assume is at most $p-1$.  Then the only possibly nonzero contributions to $\Tr(g^{p-1}(yu+s)^{p-1}\nu m')$ are of the form $\Tr(g^{p-1}(yu)^{p-1-d}s^d\nu m')$, each of which is an element of $N$ by the assumption that $\Tr((yg)^{p-1} \bullet)$ compatibly splits $N$. Thus, $u \Tr(f^{p-1}i) = v\gamma+\nu$ for $\gamma \in C$ and $\nu \in N$. It follows that $\Tr(f^{p-1}i) \in I$, see Part (3) of \Cref{prop:non-homogeneous--mixed-KR} and \Cref{rmk:iso}.
\end{proof}

We now give an example to show that, in \Cref{thm:main}, the hypothesis that there exists $u \in C$, not a zerodivisor modulo $N$, such that $u \mid g$ cannot be dispensed with.  

\begin{ex}\label{ex:u-divides-g-necessary}
Consider $I = (ty, tz, z(yx-s^2))$, the homogenization of an ideal studied by Fedder \cite{Fed83} and Singh \cite{Sin99} in order to show that the property of being F-split does not deform. Consider $C = C(y,I) = (xz, t)$ and $N = N(y,I) = (tz)$, both of which are compatibly split by the standard splitting, $\Tr(g^{p-1} \bullet)$ for $g = xyzst$.  The geometric vertex decomposition of $I$ at $y$ induces the isomorphism $C/N \rightarrow I/N$ which takes $zx$ to $z(yx-s^2)$.  Yet $\Tr((z(yx-s^2)st)^{p-1} \bullet )$ does not compatibly split $I$. Indeed, one can mimic Singh's argument in \cite[Proposition 3.1]{Sin99} to see that $S/I$ is not F-split at all.  The only hypothesis of \Cref{thm:main} not satisfied in this example is the existence of a factor of $g$ that is an element of $C$ and not a zerodivisor modulo $N$.
\end{ex}

\begin{rmk}\label{main-under-gvd-assumption}
Our primary application of \Cref{thm:main} will be in cases when $I$ is lex-compatibly geometrically vertex decomposable.  In this case and if $I$ is homogeneous, then $\psi$ constitutes an elementary G-biliaison. In the language of generalized divisors \cite{Har07}, one has a biliaison if and only if the divisor on $N(x_n,I)$ which corresponds to $I$ is linearly equivalent to the divisor of $C(x_n,I)$ plus a hypersurface section divisor on $N(x_n,I)$.
\end{rmk}

In the case of degenerate geometric vertex decomposition, a result similar to \Cref{thm:main} holds.

\begin{prop}\label{degenerate-version-of-main}
    Let $I$ be an ideal of $S$, and suppose that $\init_{x_n} (I) = C(x_n,I) \cap (N(x_n,I)+(x_n))$ is a degenerate geometric vertex decomposition. Assume that $g \in S$ has no term divisible by $x_n$ and that $\Tr((x_ng)^{p-1}\bullet)$ is a Frobenius splitting that simultaneously compatibly splits $C(x_n,I)$ and $N(x_n,I)$.  Then there is a linear form $\ell$ one of whose summands is $x_n$ such that $\Tr((\ell g)^{p-1} \bullet )$ compatibly splits $I$.  
\end{prop}

\begin{proof}
Write $C = C(x_n,I)$ and $N = N(x_n,I)$.  If $C = S$, then $I = N+(\ell)$ for some linear form $\ell$ one of whose summands is $x_n$.  Since $\init_{x_n}(\ell g) = x_ng$, $\Tr((\ell g)^{p-1} \bullet)$ splits $S$ because $\Tr((x_ng)^{p-1} \bullet)$ does by \cite[Theorem 2(1)]{Knu09}. 
    
    Because $\ell$ is a factor of $\ell g$, $\Tr((\ell g)^{p-1} \bullet )$ compatibly splits $(\ell)$.  It suffices to show that $\Tr((\ell g)^{p-1} \bullet )$ compatibly splits $N$.  Fix $n \in N$.  Because $N$ has a generating set that does not involve $x_n$, we may assume, by additivity of the trace map, that $n = x_n^km$ for some $k \geq 0$ and $m \in N$ does not involve $x_n$.  We may further assume that $k \leq p-1$.  Write $\ell  = x_n+r$, where $r$ does not involve $x_n$.  Then $\Tr((\ell g)^{p-1}n) = \sum_{i=0}^{p-1}\Tr(\binom{p-1}{i}x_n^{i+k}g^{p-1}(r^{p-1-i}m))$, the only possibly nonzero summand of which occurs when $i+k=p-1$.  But also $\Tr((x_ng)^{p-1}r^km) \in N$ because $\Tr((x_ng)^{p-1} \bullet )$ compatibly splits $N$.

    If $C \neq S$, then, by the definition of degenerate geometric vertex decomposition, $\sqrt{N} = \sqrt{C}$.  Because $C$ and $N$ are compatibly split ideals, they are radical, and so $N = C$. Thus $I = N$, and so $I$ is compatibly split by $\Tr((\ell g)^{p-1} \bullet )$.
\end{proof}

\section{Examples}\label{sect:examples}
\subsection{Two determinantal examples}
We show two examples for the purpose of practicing the implementation of \Cref{thm:main}.

We start with a toy example, the twisted cubic. 
\begin{ex}
Let $R=k[x_0,x_1,x_2]$, let
$$X=\begin{pmatrix}
x_0 & x_1 & x_2 \\
x_1 & x_2 & x_3
\end{pmatrix}
$$
and let $I=I_2(X)$ be the ideal generated by the $2$-minors of $I$. We want to prove that $I$ is compatibly split by $\Tr(f^{p-1} \bullet )$ where
\[
f = \begin{vmatrix} x_0 & x_1\\
x_1 & x_2\end{vmatrix} \begin{vmatrix} x_1 & x_2\\
x_2 & x_3\end{vmatrix}.
\]
First we note that $I$ admits a geometric vertex decomposition of the form
\[
\init_{x_3}(I) = C \cap (N+(x_3))
\] where $C = (x_0,x_1)$ and $N=(x_0x_2-x_1^2)$. As second step, we see that $N$ admits a geometric vertex decomposition of the form
\[
\init_{x_2}(N) = C' \cap (N'+(x_2)) 
\]
with $C'=(x_0)$ and $N'=0$.  
Now $C'$ and $N'$ are compatibly split by $\Tr((x_2 g')^{p-1} \bullet)$ with $g'=x_0x_1x_3$. %\Lisa{To be consistent with the notation of the theorem maybe it should be $g'=x_0 x_1$, so that we have $\Tr((x_2 g')^{p-1}\  \bullet)$?}
%\Tricia{Agreed, changed.} 
Thus by taking $u=x_0$ and applying \Cref{thm:main}, we get that $N$ is compatibly split by by $\Tr(f'^{p-1} \bullet )$ with
\[
f'=\frac{g'}{x_0}\begin{vmatrix} x_0 & x_1\\
x_1 & x_2\end{vmatrix} = \begin{vmatrix} x_0 & x_1\\
x_1 & x_2\end{vmatrix} x_1x_3.
\] As in Part (2) of \Cref{prop:non-homogeneous--mixed-KR}, we have chosen $\begin{vmatrix} x_0 & x_1\\
x_1 & x_2\end{vmatrix}$ as our representative of the image of the class of $x_0$ under the relevant isomorphism.
%\Lisa{\[
%f'=\frac{g}{x_0}\begin{vmatrix} x_0 & x_1\\
%x_1 & x_2\end{vmatrix} = \begin{vmatrix} x_0 & x_1\\
%x_1 & x_2\end{vmatrix} x_1.
%\]}
Write $g = f'/x_3 = \begin{vmatrix} x_0 & x_1\\
x_1 & x_2\end{vmatrix} x_1$. Now $C$ is a minimal prime over $(x_0x_2-x_1^2, x_1)$ and hence also compatibly split by $\Tr((x_3g)^{p-1} \bullet)$.  Thus we may apply \Cref{thm:main} with $u = x_1$ to conclude that $I$ is compatibly split by $\Tr(f^{p-1} \bullet )$ with
\[
f=\frac{g}{x_1}\begin{vmatrix} x_1 & x_2\\
x_2 & x_3\end{vmatrix} = \begin{vmatrix} x_0 & x_1\\
x_1 & x_2\end{vmatrix} \begin{vmatrix} x_1 & x_2\\
x_2 & x_3\end{vmatrix}.
\] Here we have lifted $x_1$ to $\begin{vmatrix} x_1 & x_2\\
x_2 & x_3\end{vmatrix}$, again as in Part (2) of \Cref{prop:non-homogeneous--mixed-KR}.
\end{ex}

Now we apply the theorem to a class of determinantal ideals. \Cref{prop:maximal-minors} itself is well known and can be established directly via \Cref{prop-intersect-decompose}, as is done in \cite{Knu09}.

For a positive integer $n$, let $X_n$ denote the $2 \times n$ generic matrix $\begin{pmatrix}x_{11} & \cdots & x_{1n}\\
x_{21} & \cdots & x_{2n}
\end{pmatrix}$.

\begin{prop}\label{prop:maximal-minors}
Fix integers $2 \leq n \leq N$.  Let $R = \kappa[x_{11}, x_{21}, \ldots, x_{1N},x_{2N}]$. The ideal $I$ generated by the $2$-minors of $X_n$ is compatibly split by $\Tr(f^{p-1} \bullet )$ where \[
f = x_{21}\begin{vmatrix} x_{11} & x_{12}\\
x_{21} & x_{22}\end{vmatrix} \begin{vmatrix} x_{12} & x_{13}\\
x_{22} & x_{23}\end{vmatrix} \cdots \begin{vmatrix} x_{1n-1} & x_{1n}\\
x_{2n-1} & x_{2n}\end{vmatrix} x_{1n}x_{1n+1}x_{2n+1} \cdots x_{1N}x_{2N}.
\]
\end{prop}

\begin{proof}
It is well known that the $2$-minors of a $2 \times n$ generic matrix form a diagonal Gr\"obner basis, in particular a $w$-Gr\"obner basis where $w$ is the weight vector with a $1$ in position $2n$ and $0$'s elsewhere. 
Note that $I$ admits a geometric vertex decomposition of the form
\[
\init_{x_{2n}}(I) = C \cap (N+(x_{2n}))
\] where $C = (x_{11}, \ldots, x_{1n-1})$ and $N$ is the ideal generated by the $2$-minors of $X_{n-1}$.  By induction, $N$ is compatibly split by $\Tr((x_{2n}g)^{p-1} \bullet)$ where \[
g = x_{21}\begin{vmatrix} x_{11} & x_{12}\\
x_{21} & x_{22}\end{vmatrix} \begin{vmatrix} x_{12} & x_{13}\\
x_{22} & x_{23}\end{vmatrix} \cdots \begin{vmatrix} x_{1n-2} & x_{1n-1}\\
x_{2n-2} & x_{2n-1}\end{vmatrix} x_{1n-1}x_{1n}x_{1n+1}x_{2n+1} \cdots x_{1N}x_{2N}.
\] Because \[
\left(\begin{vmatrix} x_{11} & x_{12}\\
x_{21} & x_{22}\end{vmatrix}, \begin{vmatrix} x_{12} & x_{13}\\
x_{22} & x_{23}\end{vmatrix}, \cdots, \begin{vmatrix} x_{1n-2} & x_{1n-1}\\
x_{2n-2} & x_{2n-1}\end{vmatrix}, x_{1n-1}\right) \subseteq C
\] and both are ideals of height $n-1$, $C$ is a minimal prime of the left-hand ideal, which is generated by factors of $g$.  Hence, $C$ is also compatibly split by $\Tr((x_{2n}g)^{p-1} \bullet)$.  Because $\init_{x_{2n}}\left(\begin{vmatrix} x_{1n-1} & x_{1n} \\
x_{2n-1} & x_{2n}\end{vmatrix}\right) = x_{1n-1}x_{2n}$, we have that $\begin{vmatrix} x_{1n-1} & x_{1n} \\
x_{2n-1} & x_{2n}\end{vmatrix}$ is a representative of the image of the class of $x_{1n-1}$ under the isomorphism induced by the geometric vertex decomposition of $I$ at $x_{2n}$. Note also that $x_{1n-1}$ is a factor of $g$ and a nonzerodivisor modulo $N$ (which is a prime ideal).  By \Cref{thm:main}, $\Tr(f^{p-1} \bullet)$ compatibly splits $I$.
\end{proof}

\subsection{Lower bound cluster algebras}

We will next use \Cref{thm:main} to show that the defining ideals of lower bound cluster algebras are compatibly split, recovering ~\cite[Theorem 3.7]{BMRS15} (and including the case of frozen variables). For background and notation, we refer the reader to~\cite{BMRS15}.

Over our fixed perfect field $\kappa$ of prime characteristic, let $R = \kappa[x_1, \ldots, x_n, y_1, \ldots, y_n]$, equipped with the lexicographic ordering $<$ on the variables $y_1>\cdots>y_n>x_1>\cdots>x_n$.

For a seed $(Q,\mathbf{x})$, let $K_Q$ denote the defining ideal of the lower bound cluster algebra $\mathcal{L}(Q, \mathbf{x})$.  We recall a result from~\cite{MRZ18} characterizing a $<$-Gr\"obner basis of $K_Q$.  For $U \subseteq [k]$, let $U+1 = \{u+1 : u \in U\}$.

\begin{thm}\cite[Theorem 1.7]{MRZ18}\label{thm:KQ-Groebner}
    Let $Q$ be a quiver on $n$ vertices without loops or directed $2$-cycles.  Let $\mathbf{x} = \{x_{v_1}, \ldots, x_{v_n}\}$ denote the associated cluster variables. %and designate all of the vertices of $Q$ to be mutable.  
    The following set forms a $<$-Gr\"obner basis of $K_Q$:
    \begin{itemize}
        \item For each mutable (i.e., unfrozen) vertex $i \in [n]$, set $g_i= y_ix_i -p_i^+-p_i^-$.
        \item For each frozen vertex $i\in [n]$, set $g_i= y_ix_i-1$.
        \item For each vertex-minimal directed cycle $c$ of mutable vertices, $v_1 \rightarrow v_2 \rightarrow \cdots \rightarrow v_k \rightarrow v_{k+1} = v_1$, set \[
h_c = \sum_{\substack{U \subseteq [k]\\ U \cap (U+1) = \emptyset}} (-1)^{|U|} \left( \prod_{i \in U} \dfrac{p^+_{v_i}p^-_{v_{i+1}}}{x_{v_i}x_{v_{i+1}}}\right)\left(\prod_{i \notin U \cup (U+1)} y_{v_i}\right) - \prod_{i \in [k]} \dfrac{p^+_{v_i}}{x_{v_i}} - \prod_{i \in [k]} \dfrac{p^-_{v_i}}{x_{v_i}}.
        \]
    \end{itemize}
    Consequently, \begin{align*}
    \init_<(K_Q) &= (x_iy_i : i \in [n])+(y_{v_1}\cdots y_{v_k} :\\
    & v_1 \rightarrow v_2 \rightarrow \cdots \rightarrow v_k \rightarrow v_{k+1} = v_1 \mbox{ is a vertex-minimal cycle of mutable vertices}).
    \end{align*}
\end{thm}

\begin{prop}\label{prop:cluster}
    Let $Q$ be a quiver on $n$ vertices without loops or directed $2$-cycles.  
    The ideal $K_Q$ of $R$ is compatibly split by $\Tr(f^{p-1} \bullet)$ for $f = g_1\cdots g_n$, where $g_1,\dots, g_n$ are as in Theorem \ref{thm:KQ-Groebner}.
\end{prop}

The proof of \Cref{prop:cluster} largely consists of describing precisely an arbitrary ideal arising after executing some number of steps of geometric vertex decomposition.  Before parsing the notation, the reader may prefer to read \Cref{example:cluster}. 

\begin{proof}
%One sees from \Cref{thm:KQ-Groebner} that the requirements that exactly two arrows enter and exactly two arrows exit each vertex together with the condition that no vertices of $Q$ are frozen ensure that $K_Q$ is homogeneous.  Moreover, 
As explained in~\cite[Proposition 5.2]{KR21}, one may use the results of \Cref{thm:KQ-Groebner} to infer that $K_Q$ is geometrically vertex decomposable, with decompositions determined by the order $<$.  Specifically, all ideals obtained as links or deletions after shedding at $y_1, \ldots, y_n$ are squarefree monomial ideals, split by $\Tr(f_n^{p-1} \bullet)$ for $f_n = \prod_{i \in [n]} x_iy_i$.

Fix $j \in [n]$.  Fix a subset $D \subseteq [j]$.  We will proceed via a sequence of geometric vertex decompositions, following the ideal of the deletion when $i \in D$ and the ideal of the link when $i \notin D$.  More precisely, set $I_{D,0} = K_Q$ and $I_{D,i} = \begin{cases}
    N(y_i,I_{D,i-1}), i \in D\\
    C(y_i, I_{D,i-1}), i \in [j] \setminus D
\end{cases}$ for each $i \in [j]$.
Set \[
f_j =  \left(\prod_{i \in [j+1,n]} g_i\right) \left(\prod_{i \in [j]}x_iy_i\right).
\]  We claim that, for all choices of $D$, $\Tr(f_j^{p-1} \bullet)$ compatibly splits $I_{D,j}$.  We will proceed by induction, taking the splitting of each squarefree monomial ideal $I_{D,n}$ by $f_n$ as a base case.  Taking $[0] = \emptyset$, note that $f = f_0$, and so permitting $K_Q = I_{\emptyset,0}$ yields the desired result.

Suppose the claim holds for some $0<j \leq n$.  Fix $D \subseteq [j-1]$.  By induction, $f_j$ compatibly splits $C(y_j,I_{D,j-1}) = I_{D,j}$ and $N(y_j,I_{D,j-1}) = I_{D \cup \{j\}, j}$.  
Note also that $f_j$ can be written as $f_j = y_j\tilde{f_j}$ where $\tilde{f_j}$ has no term divisible by $y_j$. 
By~\cite[Theorem 2.1]{KMY09}, $N(y_j,I_{D,j-1})$ has a Gr\"obner basis whose only elements with leading terms divisible by any $x_i$ are $\{g_{j+1}, \ldots, g_n\}$.  It follows that $x_j$ does not divide zero modulo $N(y_j,I_{D,j-1})$. Hence, $x_j \in C(y_j,I_{D,j-1})\setminus \bigcup_{P\in \ass(N(y_j,I_{D,j-1}))}P$. Observe that $x_j$ also divides $\tilde{f_j}$. 
%Because $K_Q$ is geometrically vertex decomposable via the ordering of the variables determined by $<$, the unmixedness 
Thus, the hypotheses of \Cref{thm:main} are satisfied.  
The isomorphism $\psi$ induced by the geometric vertex decomposition of $I_{D,j-1}$ at $y_j$ sends the class of $x_j$ to the class of $g_j$.  Hence, by \Cref{thm:main}, $\Tr(f_{j-1}^{p-1} \bullet)$ compatibly splits $I_{D,j-1}$. 
\end{proof}

Though the compatible splittings of the intermediate ideals $I_{D,j}$ are not directly discussed in other papers, they can readily be inferred by combining~\cite[Theorem 2(2)]{Knu09},~\cite[Proposition 5.8]{BMRS15}, and~\cite[Proposition 5.2]{KR21}.

\begin{ex}[Markov cluster algebra]\label{example:cluster}  We now apply \Cref{prop:cluster} to a lower bound algebra of the Markov  cluster algebra. See Figure \ref{fig:Markov} below.  In \cite[Proposition 5.3]{BMRS15}, the authors show that the Markov upper cluster algebra is strongly $F$-regular in characteristics not equal to $2$ or $3$.  This is noteworthy because the Markov cluster algebra, introduced and studied in \cite{BFZ05}, demonstrates a variety of pathologies and is a standard source of counterexamples in the cluster algebra literature.   
\begin{figure}[h]
    \centering
\begin{tikzpicture}[shorten >=1pt, node distance=1.5cm, auto, thick]
    % Define nodes with circles around them
    \node[circle, draw] (1) {$x_1$}; 
    \node[circle, draw] (2) [below left of=1, yshift=-1cm] {$x_2$}; 
    \node[circle, draw] (3) [below right of=1, yshift=-1cm] {$x_3$}; 
    
    % Draw two arrows from 1 to 2 with adjusted curvature
    \path[->] 
    (1) edge [bend left=20, above] node {} (2)
         edge [bend right=20, below] node {} (2)
    
    % Draw two arrows from 2 to 3 with adjusted curvature
    (2) edge [bend left=20, above] node {} (3)
         edge [bend right=20, below] node {} (3)
    
    % Draw two arrows from 3 to 1 with adjusted curvature
    (3) edge [bend left=20, above] node {} (1)
         edge [bend right=20, below] node {} (1);
\end{tikzpicture}
    \caption{The quiver for the Markov cluster algebra. All vertices are mutable.}
    \label{fig:Markov}
\end{figure}
In this example, $g_1 = y_1x_1-x_2^2-x_3^2$, $g_2 = y_2x_2-x_1^2-x_3^2$, $g_3 = y_3x_3-x_1^2-x_2^2$, and, for the unique directed cycle $c = (x_1 \rightarrow x_2 \rightarrow x_3 \rightarrow x_1)$, $h_c = y_1y_2y_3-y_3x_1x_2-y_2x_1x_3-y_1x_2x_3-2x_1x_2x_3$.  Then $K_Q = (g_1, g_2, g_3, h_c)$, and the given generators are a Gr\"obner basis under the lexicographic term order determined by $y_1>y_2>y_3>x_1>x_2>x_3$.  The tree pictured in \Cref{fig:cluster-gvd} records the corresponding steps in geometric vertex decomposition.  

\begin{figure}[h]
	\begin{center}
	\scalebox{0.75}{
\begin{tikzpicture}
	\node (A) at (0,0) {
    $K_Q$};
	\node (B) at (-6,-2) {
 $( g_2, g_3 )$
    };
    \node (C) at (6,-2) {
 $( x_1, y_2y_3-x_2x_3, g_2, g_3 )$
    };
	\node (D) at (-9,-4) {
 $( g_3 )$
    };
 \node (E) at (9,-4) {
 $( y_3, x_1, x_2 )$
 };
 	\node (F) at (-3,-4) {
 $( x_2, g_3 )$
    };
 \node (G) at (3,-4) {
 $( x_1, g_3 )$
 };
  \node (H) at (-1.5,-6) {
 $( x_2, x_3 )$
 };
 \node (I) at (1.5,-6) {
 $( x_1 )$
 };
 \node (J) at (-4.5,-6) {
 $( x_2 )$
 };
  \node (K) at (4.5,-6) {
 $( x_1, x_3 )$
 };
 \node (L) at (-7.5,-6) {
 $( x_3 )$
 } ;
 \node (M) at (7.5,-6) {
 $( x_1, x_2 )$
 };
 \node (N) at (-10.5,-6) {
 $( 0 )$
 };
  \node (O) at (10.5,-6) {
 $( 1 )$
 };
	\draw[->,thick] (A) -- node[above]{$\del_{y_1}$} (B) ;
    \draw[->,thick] (A) -- node[above]{$\lk_{y_1}$} (C) ;
    \draw[->,thick] (B) -- node[above]{$\lk_{y_2}$} (F) ;
    \draw[->,thick] (B) -- node[above]{$\del_{y_2}${\color{white} $^{--}$}} (D) ;
    \draw[->,thick] (C) -- node[above]{$\lk_{y_2}$} (E) ;
    \draw[->,thick] (C) -- node[above]{$\del_{y_2}${\color{white} $^{--}$}} (G) ;
    \draw[->,thick] (D) -- node[right]{$\lk_{y_3}$} (L) ;
    \draw[->,thick] (D) -- node[left]{$\del_{y_3}$} (N) ;
    \draw[->,thick] (F) -- node[right]{$\lk_{y_3}$} (H) ;
    \draw[->,thick] (F) -- node[left]{$\del_{y_3}$} (J) ;
    \draw[->,thick] (G) -- node[left]{$\del_{y_3}$} (I) ;
    \draw[->,thick] (G) -- node[right]{$\lk_{y_3}$} (K) ;
    \draw[->,thick, color = blue] (E) -- node[left]{$\del_{y_3}$} (M) ;
    \draw[->,thick, color = blue] (E) -- node[right]{$\lk_{y_3}$} (O) ;
\end{tikzpicture}}
	\end{center}
\caption{Tree recording geometric vertex decomposition from $K_Q$, where $\del_{y_i}$ and $\lk_{y_i}$ denote application of the operator $N(y_i, \text{---})$ and $C(y_i,\text{---})$, respectively.  All geometric vertex decompositions recorded by this tree are nondegenerate except for $( y_3, x_1, x_2) = ( 1 ) \cap (( x_1, x_2) + ( y_3 ))$, which is highlighted with blue arrows.}
\label{fig:cluster-gvd}
\end{figure}
Because all leaves of the geometric vertex decomposition tree are squarefree monomials, $\Tr(f_3^{p-1} \bullet)$ simultaneously compatibly splits all of these ideals for $f_3 = y_1y_2y_3x_1x_2x_3$. 

Set $f_2 = y_1y_2x_1x_2g_3$.  By \Cref{thm:main}, $\Tr(f_2^{p-1} \bullet)$  simultaneously compatibly splits all ideals which are parents of leaves except possibly $( y_3, x_1, x_2 )$.  A straightforward intersect-decompose argument shows that indeed $\Tr(f_2^{p-1} \bullet)$ also compatibly splits $( y_3, x_1, x_2 )$.  As a caution, note that $\Tr(f_2^{p-1} \bullet)$ is not the splitting recommended by \Cref{degenerate-version-of-main}.  %It is not the case that a splitting constructed from one instance of a geometric vertex decomposition will necessarily compatibly split all ideals in that level of the geometric vertex decomposition tree.

Set $f_1 = y_1x_1g_2g_3$.  By \Cref{thm:main}, $\Tr(f_1^{p-1} \bullet)$ simultaneously compatibly splits both $( g_2, g_3 )$ and $( x_1, y_2y_3-x_2x_3, g_2, g_3 )$, and $\Tr((g_1g_2g_3)^{p-1} \bullet)$ compatibly splits $K_Q$.
    \end{ex}

See \cite{DGK+} for examples in other contexts where repeated use of \Cref{thm:main} can be used to build desired splittings.

\section{Frobenius splitting of Li's double determinantal ideals: Maximal minors}\label{sect:DDV}

In this section, we construct Frobenius splittings of the form $\text{Tr}(f^{p-1}\bullet)$ for double determinantal ideals in the case of maximal minors.  Double determinantal ideals, introduced by Li, are the defining ideals of a special class of Nakajima quiver varieties, which arise in the study of bases of cluster algebras (see \cite{IL25}). Double determinantal ideals are lex-compatibly geometrically vertex decomposable (by interpreting the inductive argument given in \cite{FK20}, which is phrased in terms of elementary G-biliaison, through the lens of \cite{KR21}).  Our construction of the splittings $\text{Tr}(f^{p-1}\bullet)$ is guided by the proof of \Cref{thm:main} but cannot be obtained as a consequence of \Cref{thm:main} in a natural way. We include this example to show that, in the realm of Frobenius splittings and lex-compatibly geometrically vertex decomposable ideals, there is a phenomenon that we can observe but not yet explain.  

Fix $r, m, n \geq 1$ and let $[n] = \{1, \ldots, n\}$. Let $X_k = (x^{(k)}_{i,j})$ with $1 \leq k \leq r$ be $m \times n$ matrices of distinct indeterminates and let $S =\kappa[x^{(k)}_{i,j} \mid i \in [m], j \in [n], k \in [r]]$ be the standard graded polynomial ring in the indeterminates that appear in the matrices $X_k$ over $\kappa$.

\begin{defn}
Let \(H\) be the horizontal concatenation of $X_1,\ldots,X_r$, i.e., $H$ is the \(m \times rn\) matrix $H =\begin{pmatrix}  X_1 \cdots X_r\\
\end{pmatrix}$,
and let \(V\) be their vertical concatenation, i.e., $V$ is the \(rm \times n\) matrix $V =
\begin{pmatrix}
X_{1} \\
\vdots \\
X_{r} \\
\end{pmatrix}$.  The ideal $I=I_s(H)+I_t(V)$ is called a \newword{double determinantal ideal}, and the variety cut out by $I$ is called a \newword{double determinantal variety}.
\end{defn}

In~\cite{FK20}, Fieldsteel and Klein show that the natural generators of a double determinantal ideal form a diagonal Gr\"obner basis, i.e., a Gr\"obner basis with respect to any term order under which the leading term of the determinant of any submatrix of $H$ or of $V$ is the product of terms along the main diagonal of that submatrix. In their proof, they produce a series of geometric vertex decompositions from a double determinantal ideal to an ideal generated by variables, where all the intermediate ideals in the decomposition are the analogues of one-sided mixed ladder determinantal ideals in the double determinantal setting. 

Our goal in this section is to prove that double determinantal ideals with $s=t=m=n$ are compatibly split.  The case $r = 1$ is the case of a classical determinantal ideal.  With $r \geq 2$ arbitrary, consider the  double determinantal ideal $I=I_n(H)+I_n(V)$ of maximal minors. Consider the following matrices: If $r$ is even, set
$$D^*= \begin{pmatrix}
    X_1 & X_2 & \mathbf{0}& \cdots& \mathbf{0}\\
    \mathbf{0}& X_3& X_4&\ddots& \vdots\\
    \vdots&\ddots& \ddots& \ddots&\mathbf{0}\\
    \mathbf{0}&\cdots&\mathbf{0}& X_{r-1}& X_{r}&\\
\end{pmatrix} \qquad
D_*= \begin{pmatrix}
    X_1 &  \mathbf{0}& \cdots& \mathbf{0}\\
    X_2 & X_3&\ddots& \vdots\\
    \mathbf{0}&X_4& \ddots& \mathbf{0}\\
    \vdots& &\ddots& X_{r-1}\\
    \mathbf{0}&\cdots& \mathbf{0}&X_{r}\\
    \
\end{pmatrix}
$$ And if $r$ is odd, set

$$D^*= \begin{pmatrix}
    X_1 & X_2 & \mathbf{0}& \cdots& \mathbf{0}\\
    \mathbf{0}& X_3& X_4&\ddots& \vdots\\
    \vdots&\ddots& \ddots& \ddots&\mathbf{0}\\
    &&& X_{r-2}& X_{r-1}&\\
    \mathbf{0}&\cdots&\cdots&\mathbf{0}& X_r
\end{pmatrix} \qquad
D_*= \begin{pmatrix}
    X_1 &  \mathbf{0}& \cdots&\cdots& \mathbf{0}\\
    X_2 & X_3&\ddots&& \vdots\\
    \mathbf{0}&X_4& \ddots&\ddots& \vdots\\
    \vdots& &\ddots& X_{r-2}&\mathbf{0}\\
    \mathbf{0}&\cdots& \mathbf{0}&X_{r-1}&X_r\
\end{pmatrix}$$

Define $\delta^s$ to be the diagonal of $D^*$ starting from $x^{(1)}_{1,s+1}$ for $s=1,\ldots,n-1$ and $\delta_s$ to be the diagonal of $D_*$ starting from $x^{(1)}_{s+1,1}$ for $s=1,\ldots,n-1$. Let $\Delta^s$ and $\Delta_s$ be their corresponding minors.

Before proceeding with the theorem statement and proof, we provide an example in order to practice the notation.

\begin{ex}
    With $m=n=s=t=r=3$, setting $x_{i,j}^{(2)} = y_{ij}$ and $x_{i,j}^{(3)} = z_{ij}$, write \[
H = \begin{pmatrix}
    x_{11} & x_{12} & x_{13} & y_{11} & y_{12} & y_{13} & z_{11} & z_{12} & z_{13}\\
    x_{21} & x_{22} & x_{23} & y_{21} & y_{22} & y_{23} & z_{21} & z_{22} & z_{23}\\
    x_{31} & x_{32} & x_{33} & y_{31} & y_{32} & y_{33} & z_{31} & z_{32} & z_{33}
\end{pmatrix}
    \] and \[
D^* = \begin{pmatrix}
    x_{11} & x_{12} & x_{13} & y_{11} & y_{12} & y_{13} \\
    x_{21} & x_{22} & x_{23} & y_{21} & y_{22} & y_{23} \\
    x_{31} & x_{32} & x_{33} & y_{31} & y_{32} & y_{33} \\
    0 & 0 & 0 & z_{11} & z_{12} & z_{13}\\
    0 & 0 & 0 & z_{21} & z_{22} & z_{23}\\
    0 & 0 & 0 & z_{31} & z_{32} & z_{33}
\end{pmatrix}.
    \]

    Then
    
    \[
    \Delta^1 = \begin{vmatrix} 
    x_{12} & x_{13} & y_{11} & y_{12} & y_{13}\\
    x_{22} & x_{23} & y_{21} & y_{22} & y_{23}\\
    x_{32} & x_{33} & y_{31} & y_{32} & y_{33}\\
    0 & 0 & z_{11} & z_{12} & z_{13}\\
    0 & 0 & z_{21} & z_{22} & z_{23}
    \end{vmatrix}, \qquad \mbox{ and } \qquad \Delta^2 = \begin{vmatrix} 
    x_{13} & y_{11} & y_{12} & y_{13}\\
    x_{23} & y_{21} & y_{22} & y_{23}\\
    x_{33} & y_{31} & y_{32} & y_{33}\\
    0 & z_{11} & z_{12} & z_{13}
    \end{vmatrix}.
    \]
\end{ex}

\begin{Theorem}\label{thm:DDVSplitting}
     Assume $s=t=m=n$. Define $f\in S$ to be the polynomial $f= \prod_{k=1}^r \det X_k \prod_{u=1}^{n-1} \Delta_u \Delta^u$.
Then $I=I_n(H)+I_n(V)$ is a Knutson ideal of $f$. In particular, $\varphi_f=\Tr(f^{p-1} \bullet)$ compatibly splits $I$.
\end{Theorem}

\begin{proof}
    For $i=1,\ldots, r-2$ define $H_{i,i+1,i+2}$ to be the submatrix $\begin{pmatrix}  X_i X_{i+1}X_{i+2} \end{pmatrix}$ of $H$. Using the standard formula for the height of a determinantal ideal, $\hgt (I_n(H_{i,i+1,i+2}))=2n+1$. Consider the ideal \[
    J_i= \left(\det X_i,\det X_{i+1},\det X_{i+2}, \Delta^1, \Delta_1, \ldots, \Delta^{n-1},\Delta_{n-1}\right).
    \]
    With respect to a diagonal term order, the leading terms of the given generators of $J_i$ are relatively prime.  Hence, $J_i$ is a complete intersection of height \[
    3+2(n-1)=2n+1=\hgt (I_n(H_{i,i+1,i+2})).
    \]
    Moreover, using Laplace expansion on the $\Delta$'s, we see that $J_i \subseteq I_n( H_{i,i+1,i+2})$.
    Thus $I_n( H_{i,i+1,i+2})$ is a minimal prime over $J_i$. But $J_i$ is a Knutson ideal of $f$ because its generators are factors of $f$, and so $I_n( H_{i,i+1,i+2})$ is also a Knutson ideal of $f$. 
    
    Therefore the sum \[
    I_\Sigma=I_n( H_{1,2,3})+\cdots+I_n( H_{r-2,r-1,r})
    \]
    is a Knutson ideal of $f$. Observe that $I_{\Sigma}$ contains the complete intersection generated by minors corresponding to the $rn-n+1$ distinct diagonals of $H$ and  is contained in $I_n(H)$. Since both this complete intersection and also $I_n(H)$ have height $rn-n+1$, $I_{\Sigma}$ has height $rn-n+1$ as well. Hence $I_n(H)$ is a minimal prime over the Knutson ideal $I_{\Sigma}$ and so is also a Knutson ideal of $f$ .
    
    The argument that $I_n(V)$ is a Knutson ideal of $f$ is analogous. Therefore,  the sum $I=I_n(H)+I_n(V)$ is also a Knutson ideal of $f$, as desired. 
\end{proof}

The reason why \Cref{thm:main} cannot be used to obtain \Cref{thm:DDVSplitting} inductively using the geometric vertex decomposition from \cite{FK20} is that one would be led to construct a splitting $\Tr(g^{p-1} \bullet)$ such that each factor of $g$ is in the ideal of the deletion.  In the notation of \Cref{thm:main}, one arrives at a splitting failing the hypothesis that there exist $u \mid g$ which is a nonzerodivisor modulo $N$.  This occurs already for $n=2$ and $r=3$. However, for a fixed stage in the inductive construction, with $\init_{x_{ij}^{(k)}} I = C \cap (N+(x_{ij}^{(k)}))$ and the splitting $\Tr(g^{p-1} \bullet)$, there is always a factor $u$ of $g$ which may be expressed as a sum of multiples of elements in $C$ each of which is a nonzerodivisor modulo $N$.  Applying the construction in \Cref{thm:main} to each of these elements individually, and possibly modifying by an element of $N$, produces the splittings of \Cref{thm:DDVSplitting}.  It is in this sense that the construction in \Cref{thm:DDVSplitting} is suggested by the construction of \Cref{thm:main}.

\begin{ex}
Let \[
f = \begin{vmatrix} x_{11} & x_{12}\\
x_{21} & x_{22} \end{vmatrix}\cdot \begin{vmatrix} y_{11} & y_{12}\\
y_{21} & y_{22} \end{vmatrix} \cdot \begin{vmatrix} x_{21} & x_{22} & 0\\
y_{11} & y_{12} & z_{11}\\
y_{21} & y_{22} & z_{21}\end{vmatrix} \cdot \begin{vmatrix} x_{12} & y_{11}\\
x_{22} & y_{21} 
\end{vmatrix} \cdot z_{11}z_{12}z_{22} 
\]
Set \[
H = \begin{pmatrix} x_{11} & x_{12} & y_{11} & y_{12} & z_{11}\\
x_{21} & x_{22} & y_{21} & y_{22} & z_{21} \end{pmatrix}, \qquad  V = \begin{pmatrix} x_{11} & x_{12}\\
x_{21} & x_{22}\\
y_{11} & y_{12}\\
y_{21} & y_{22}\\
z_{11} & z_{12}
\end{pmatrix}, \mbox{ and } \qquad V^- = \begin{pmatrix} x_{11} & x_{12}\\
x_{21} & x_{22}\\
y_{11} & y_{12}\\
y_{21} & y_{22}
\end{pmatrix}.
\]
Consider the ladder double determinantal ideal $I = I_2(H)+I_2(V)$.  Set $N = I_2(H)+I_2(V^-)$ and $C = (x_{11}, x_{21}, y_{11}, y_{21}, x_{12}z_{21}-z_{11}x_{22})$. Consider the geometric vertex decomposition $\init_{z_{12}}(I) = C \cap (N+(z_{12}))$.  Suppose we have established that $\Tr(f^{p-1} \bullet)$ compatibly splits the ladder double determinantal ideals $C$ and $N$, both of which are prime. Because every factor of $f$ that is an element of $C$ is also an element of $N$, we cannot apply \Cref{thm:main}.  However, we may view the factor $\Delta = x_{12}y_{21}-x_{22}y_{11}$ as a sum of multiples of the elements $y_{21}, y_{11} \in C \setminus N$.  Applying the isomorphism corresponding to the geometric vertex decomposition to each of $y_{21}$ and $y_{11}$ separately and substituting the outputs, $\begin{vmatrix} y_{21} & y_{22} \\
z_{11} & z_{12}
\end{vmatrix}$ and $\begin{vmatrix} y_{11} & y_{12} \\
z_{11} & z_{12}
\end{vmatrix}$, respectively, into $\Delta$ yields \[
\Delta' = x_{12}\begin{vmatrix} y_{21} & y_{22} \\
z_{11} & z_{12}
\end{vmatrix}-x_{22}\begin{vmatrix} y_{11} & y_{12} \\
z_{11} & z_{12}
\end{vmatrix} = \begin{vmatrix} x_{12} & y_{11} & y_{12} \\
x_{22} & y_{21} & y_{22}\\
0 & z_{11} & z_{12}
\end{vmatrix}.
\] Then $\Tr((\Delta'f/\Delta)^{p-1} \bullet)$ is the expected splitting which compatibly splits $I$, which the reader may verify via  \Cref{thm:DDVSplitting} and \cite[Theorem 2]{Knu09}.
\end{ex}
 
The authors would be very interested to see a theorem that refines or extends \Cref{thm:main} and accounts for the construction of the splittings in this section.  

\section{Acknowledgments}

The authors thank Allen Knutson and Matteo Varbaro for helpful conversations. They also thank Sara Faridi, Martina Juhnke-Kubitzke, Haydee Lindo, Elisa Postinghel, and Alexandra Seceleanu, who, together with Elisa Gorla, organized the workshop Women in Commutative Algebra II, where this project began. The authors thank Mayada Shahada for discussions concerning the content of the paper and Greg Smith for advice on Macaulay2.

\bibliographystyle{amsalpha}
\bibliography{refs}

\end{document}